\documentclass{article}
\usepackage[cmex10]{amsmath}
\usepackage{amssymb,amsthm,mathtools}
\usepackage[pdftex]{graphicx}
\usepackage[colorlinks=true,allcolors=blue]{hyperref}\hypersetup{pdfcreator={LaTeX2e}}
\newcommand*{\doi}[1]{doi:\href{http://dx.doi.org/#1}{\nolinkurl{#1}}}
\newcommand*{\xurl}[1]{URL \url{#1}}
\newtheorem{theorem}{Theorem}

\newcommand*{\abs}[1]{\left\lvert#1\right\rvert}
\newcommand*{\norm}[1]{\left\lVert#1\right\rVert}
\newcommand*{\intcc}[1]{\left[#1\right]}
\newcommand*{\bigpar}[1]{\bigl(#1\bigr)}

\begin{document}
\title{$N$ derivatives are necessary for order $N+1$ convergence in quadrature: a converse result}
\date{}
\author{Jeffrey Tsang\\\href{mailto:jeffrey.tsang@ieee.org}{\normalsize jeffrey.tsang@ieee.org}}
\maketitle
\begin{abstract}
Results on the error bounds of quadrature methods are well known --- most state that if the method has degree $N$, and the integrand has $N$ derivatives, then the error is order $N+1$. We prove here a converse: that if the integrand fails to have $N$ derivatives, even only at a finite number of points, no method, regardless of its degree, can guarantee convergence more than order $N$. Even if the integrand fails to have $N$ derivatives at just 3 (for even $N$, 2) points, no method can produce order more than $N+1$ convergence. This is done by an adversarial proof: we explicitly construct the functions that exhibit such error; simple splines turn out to suffice.
\end{abstract}

\section{Introduction.}
It is well known that a quadrature method of degree $N$ has an error bound of order $N+1$; many proofs rely on Taylor's theorem, which then requires the existence of the $N$th derivative. Of course, it can be asked what happens if the $N$th derivative fails to exist.

The question was first investigated in \cite{stroud}, which proved error bounds using lower derivatives. These bounds do not attain the full order allotted by the degree of the method; instead they are consigned to no more than the number of derivatives plus one.

Recently attention has returned to the issue of integration of insufficiently nice functions, where \cite{engel} proved a few error bounds using lower order derivatives than the degrees of the methods involved. \cite{cruz} looked into the trapezoidal rule and Simpson's rule and found that, for merely once differentiable functions, the trapezoidal rule had a better bound than Simpson's rule.

A different direction is to optimize quadrature rules under the constraint that the integrand is only once differentiable; this work is led by \cite{ujevic} and there are quite a few papers on the subject. The optimal rules indeed differ from the established Gauss-Legendre rules. There is also a focus on Simpson's inequality and its applications, as demonstrated by \cite{huy} who proved a bound using an arbitrary order derivative.

One unifying theme in all these results is the inability to achieve $N+2$th order convergence with only the $N$th derivative. We settle this question once and for all, by proving this impossibility. In other words, $N$ derivatives are not only sufficient (in conjunction with a suitable method), but also a \emph{necessary} condition of achieving $N+1$th order convergence.

We prove the contrapositive by exhibiting functions, exactly $N$ times differentiable, that bound the error of any method from below at $N+1$th order.

\section{Results.}
For space, we use the symbol $k-1$ instead of $N$ in the ensuing discussion.

\begin{theorem}
\label{t1}
\hypertarget{t1}{Let} the region of integration $[a,b]$ be specified, and $k\geq1$ be a fixed natural number. Let $\{I_n\}$ be any sequence of internal quadrature methods, with
\[I_n(f)=\sum_{i=1}^nw_{n,i}f(x_{n,i}),\quad a\leq x_{n,1}<x_{n,2}<\cdots<x_{n,n}\leq b\text{,}\]
with $w_{n,i},x_{n,i}$ fixed. Define the exact integral to be
\[I(f)=\int_a^bf(x)\,\mathrm{d}x\text{.}\]
Then there exists a sequence of functions $\{f_n\}$ along with a sequence of finite sets $\{S_n\}$ satisfying
\begin{itemize}
\item $S_n\subset[a,b]$, and has size $\abs{S_n}\leq4n-3$
\item $f_n\in C^{k-1}[a,b]\setminus C^k[a,b]$
\item $f_n\in C^{\omega}([a,b]\setminus S_n)$
\item $\displaystyle\norm{f_n^{(k-1)}}_{\infty,[a,b]}\leq\frac{(b-a)}{4}k!$
\item $\displaystyle\norm{f_n^{(k)}}_{\infty,[a,b]\setminus S_n}=k!$
\end{itemize}
such that the sequence of integration errors of $I_n$ on $f_n$ has order at most $k$, or in other words
\[\limsup_{n\rightarrow\infty}\frac{\abs{I_n(f_n)-I(f_n)}}{n^k}>0\text{.}\]
\end{theorem}
\begin{proof}
By adversarial construction.

Let $f_n$ be a $k$th order spline such that $f_n(x_{n,i})=0$ for all $i$, and have maximal integral with unit leading coefficient. First, for $1\leq i\leq n-1$ define $\Delta x_i=x_{n,i+1}-x_{n,i}$ and the quarter-interpolated points
\begin{alignat*}{3}
\bar{x}_{n,i+1/4}&=\frac{3x_{n,i}}{4}&&+\hspace{0.049in}\frac{x_{n,i+1}}{4}&&=x_{n,i}+\frac{\Delta x_i}{4}\\
\bar{x}_{n,i+1/2}&=\hspace{0.039in}\frac{x_{n,i}}{2}&&+\hspace{0.049in}\frac{x_{n,i+1}}{2}&&=x_{n,i}+\frac{\Delta x_i}{2}\\
\bar{x}_{n,i+3/4}&=\hspace{0.039in}\frac{x_{n,i}}{4}&&+\frac{3x_{n,i+1}}{4}&&=x_{n,i}+\frac{3\Delta x_i}{4}\text{.}
\end{alignat*}
Then explicitly, if $k$ is odd,
\begin{align*}f_n&=\begin{cases}
-(x-x_{n,1})^k,\;&x\in\intcc{a,x_{n,1}}\\
(x-x_{n,i})^k,\;&x\in\intcc{x_{n,i},\bar{x}_{n,i+1/4}}\\
\left(x-\bar{x}_{n,i+1/2}\right)^k+2\left(\frac{\Delta x_i}{4}\right)^k,\;&x\in\intcc{\bar{x}_{n,i+1/4},\bar{x}_{n,i+1/2}}\\
-\left(x-\bar{x}_{n,i+1/2}\right)^k+2\left(\frac{\Delta x_i}{4}\right)^k,\;&x\in\intcc{\bar{x}_{n,i+1/2},\bar{x}_{n,i+3/4}}\\
-(x-x_{n,i+1})^k,\;&x\in\intcc{\bar{x}_{n,i+3/4},x_{n,i+1}}\\
(x-x_{n,n})^k,\;&x\in\intcc{x_{n,n},b}
\end{cases}
\intertext{for all valid $1\leq i\leq n-1$; and if $k$ is even,}
f_n&=\begin{cases}
(x-x_{n,1})^k,\;&x\in\intcc{a,\bar{x}_{1+1/4}}\\
-\left(x-\bar{x}_{n,i+1/2}\right)^k+2\left(\frac{\Delta x_i}{4}\right)^k,\;&x\in\intcc{\bar{x}_{n,i+1/4},\bar{x}_{n,i+3/4}}\\
(x-x_{n,j})^k,\;&x\in\intcc{\bar{x}_{(j-1)+3/4},\bar{x}_{n,j+1/4}}\\
(x-x_{n,n})^k,\;&x\in\intcc{\bar{x}_{(n-1)+3/4},b}
\end{cases}
\end{align*}
for all valid $1\leq i\leq n-1$ and $2\leq j\leq n-2$.

An easier way to understand this is to consider the ``basic unit'' of structure between two $x_{n,i}$, shown in Figure \ref{base} between $x_{n,i}=0$ and $x_{n,i+1}=4$. The concept is to splice 4 (shifted) copies of the curve $x^k$ together such that the two outer pieces are convex and the inner pieces concave. If $k$ is even, then it is analytic at the endpoints and midpoint.

\begin{figure}[h]
\centering\includegraphics[width=\textwidth,keepaspectratio]{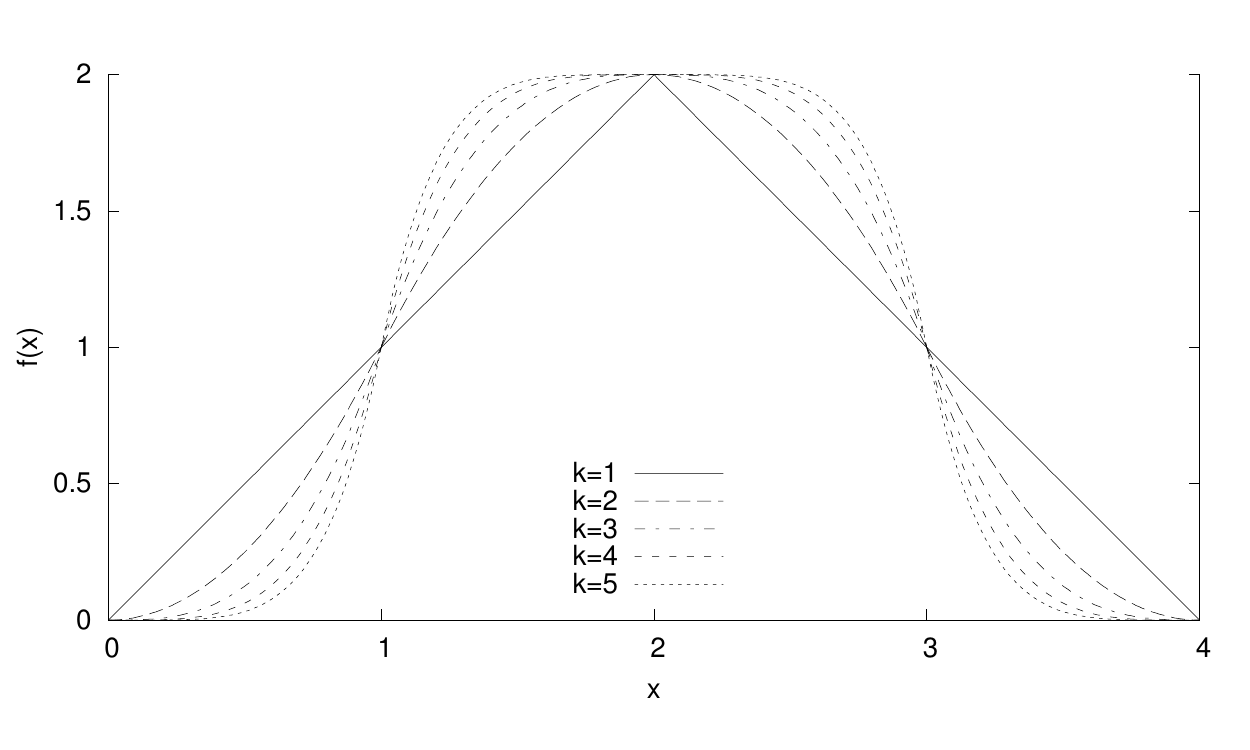}
\caption{A basic unit of $f_n$ for various values of $k$, shown for $x_{n,i}=0$ and $x_{n,i+1}=4$}
\label{base}
\end{figure}

Now it is clear that $f_n$ is $C^{k-1}[a,b]$, is analytic except at most at the points
\[S_n=\left\{x_{n,i},\frac{3x_{n,i}+x_{n,i+1}}{4},\frac{x_{n,i}+x_{n,i+1}}{2},\frac{x_{n,i}+3x_{n,i+1}}{4}\right\}_{i=1}^{n-1}\bigcup\;\{x_{n,n}\}\text{,}\]
a set of size $4n-3$ (only $2n-2$ points if $k$ is odd). The $k-1$th derivative of $f_n$ exists and is piecewise linear, with slope $\pm k!$; the length of a piece is bounded above by $(b-a)/4$, possible only if all points $x_{n,i}$ are at endpoints, thus the derivative is bounded by that product. If $k$ is odd the piece can have length $(b-a)/2$, but its midpoint is at 0 and that does not change the bound. The $k$th derivative is the piecewise constant function $\pm k!$ where it exists.\pagebreak

Note that $I_n(f_n)=0$ by construction; it remains only to compute $I(f_n)$, which is easily done by considering the basic unit:
\begin{align*}
\int_{x_{n,i}}^{x_{n,i+1}}\!\!f_n(x)\,\mathrm{d}x&
\begin{cases}\displaystyle=\int_0^{\frac{\Delta x_i}{4}}\!x^k\,\mathrm{d}x+\int_{-\frac{\Delta x_i}{4}}^0\!x^k+2\left(\textstyle\frac{\Delta x_i}{4}\right)^k\!\mathrm{d}x\\
\displaystyle\!\!\quad{}+\int_0^{\frac{\Delta x_i}{4}}\!(-x)^k+2\left(\textstyle\frac{\Delta x_i}{4}\right)^k\!\mathrm{d}x+\int_{-\frac{\Delta x_i}{4}}^0\!(-x)^k\,\mathrm{d}x\!\!&\text{ if $k$ is odd}\\
\displaystyle=\int_0^{\frac{\Delta x_i}{4}}\!x^k\,\mathrm{d}x+\int_{-\frac{\Delta x_i}{4}}^{\frac{\Delta x_i}{4}}\!(-x)^k+2\left(\textstyle\frac{\Delta x_i}{4}\right)^k\!\mathrm{d}x\\
\displaystyle\!\!\quad{}+\int_{-\frac{\Delta x_i}{4}}^0\!x^k\,\mathrm{d}x&\text{ if $k$ is even}
\end{cases}\displaybreak[0]\\
&\hspace{0.12in}=\int_{-\frac{\Delta x_i}{4}}^{\frac{\Delta x_i}{4}}\!2\left(\frac{\Delta x_i}{4}\right)^k\!\mathrm{d}x\quad\bigpar{\forall k}\\
&\hspace{0.12in}=\frac{(\Delta x_i)^{k+1}}{4^k}\text{.}
\end{align*}
The pieces for the endpoints are easy:
\begin{align*}
\int_{a}^{x_{n,1}}f_n(x)\,\mathrm{d}x=\int_0^{x_{n,1}-a}x^k\,\mathrm{d}x&=\frac{(x_{n,1}-a)^{k+1}}{k+1}\\
\int_{x_{n,n}}^{b}f_n(x)\,\mathrm{d}x=\int_0^{b-x_{n,n}}x^k\,\mathrm{d}x&=\frac{(b-x_{n,1})^{k+1}}{k+1}\text{.}
\end{align*}
Thus $I(f_n)$, and the integration error, is simply
\[\frac{(x_{n,1}-a)^{k+1}}{k+1}+\sum_{i=1}^{n-1}\frac{(x_{n,i+1}-x_{n,i})^{k+1}}{4^k}+\frac{(b-x_{n,1})^{k+1}}{k+1}\text{.}\]
This is straightforward to minimize by picking equally spaced points (we ignore the issue of picking points closer to the endpoints, which is unimportant), whereby the error is bounded below by
\[\sum_{i=0}^{n}C\left(\frac{b-a}{n+1}\right)^{k+1}=C(n+1)\left(\frac{b-a}{n+1}\right)^{k+1}\in\Theta(n^{-k})\text{.}\]
Therefore the method has at most order $k$ convergence, and the theorem is proved.
\end{proof}

There are two main ways of conceptualizing a ``sequence of quadrature methods''. The first is an arbitrary-order family of methods, for example the Newton-Cotes or Gauss-Legendre, with an increasing number of points. The second is to take a fixed quadrature method, and form a composite rule with $n$ subintervals, for all $n$.

Note that the construction obviates all dependence on the weights chosen by the quadrature method. The restriction to internal quadrature methods is not necessary: a point outside $[a,b]$ can easily be made to evaluate to 0, with the result that the error bound becomes worse.

Nor do methods that rely on the derivatives of the function escape this fate. It is straightforward to check that the $k-1$th and lower derivatives of $f_n$ are all 0 at each $x_{n,i}$, and so unless the method uses the $k$th derivative (which doesn't exist in the first place), the construction stands valid.

The size of $S_n$ grows linearly in $n$, which disallows attempts to sidestep the theorem by using less points for evaluation than the number of points known to be bad.

We may then ask the question: if the number of bad points is finitely bounded, can something more be salvaged? The idea is, say with a composite rule, with a sufficient number of subintervals, to confine all the bad points to specific subinterval(s), whereby all other subintervals may achieve their degree-based bound.

The answer is a qualified yes. The proof of Theorem \ref{t1} can be easily modified to use a very small number of bad points, but this only allows improvement by 1 order, to $k+1$. We shall reuse the setting in the \hyperlink{t1}{preface} of the preceding theorem.

\begin{theorem}
\label{t2}
There exists a sequence of functions $\{f_n\}$ along with a sequence of finite sets $\{S_n\}$ satisfying
\begin{itemize}
\item $S_n\subset[a,b]$, and has size $\abs{S_n}=5$ if $k$ is odd, or $\abs{S_n}=4$ if $k$ is even
\item $f_n\in C^{k-1}[a,b]\setminus C^k[a,b]$
\item $f_n\in C^{\omega}([a,b]\setminus S_n)$
\item $\displaystyle\norm{f_n^{(k-1)}}_{\infty,[a,b]}\leq\frac{(b-a)}{4}k!$
\item $\displaystyle\norm{f_n^{(k)}}_{\infty,[a,b]\setminus S_n}=k!$
\end{itemize}
such that the sequence of integration errors of $I_n$ on $f_n$ has order at most $k+1$, or in other words
\[\limsup_{n\rightarrow\infty}\frac{\abs{I_n(f_n)-I(f_n)}}{n^{k+1}}>0\text{.}\]
\end{theorem}
\begin{proof}
By adversarial construction.

Let $i$ be an index such that $\Delta x=x_{n,i+1}-x_{n,i}$ is maximized. Let $f_n$ be a $k$th order spline such that $f_n=0$ outside of the open interval $(x_{n,i},x_{n,i+1})$, and have maximal integral with unit leading coefficient. Explicitly,
\begin{align*}
f_n&=\begin{cases}
(x-x_{n,i})^k,\;&x\in\intcc{x_{n,i},x_{n,i}+\frac{\Delta x}{4}}\\
\left(x-\frac{\Delta x}{2}\right)^k+2\left(\frac{\Delta x}{4}\right)^k,\;&x\in\intcc{x_{n,i}+\frac{\Delta x}{4},x_{n,i}+\frac{\Delta x}{2}}\\
-\left(x-\frac{\Delta x}{2}\right)^k+2\left(\frac{\Delta x}{4}\right)^k,\;&x\in\intcc{x_{n,i}+\frac{\Delta x}{2},x_{n,i}+\frac{3\Delta x}{4}}\\
-(x-x_{n,i+1})^k,\;&x\in\intcc{x_{n,i}+\frac{3\Delta x}{4},x_{n,i+1}}\\
0\quad&\text{otherwise}
\end{cases}\text{ if $k$ is odd,}\\
f_n&=\begin{cases}
(x-x_{n,i})^k,\;&x\in\intcc{x_{n,i},x_{n,i}+\frac{\Delta x}{4}}\\
-\left(x-\frac{\Delta x}{2}\right)^k+2\left(\frac{\Delta x}{4}\right)^k,\;&x\in\intcc{x_{n,i}+\frac{\Delta x}{4},x_{n,i}+\frac{3\Delta x}{4}}\\
(x-x_{n,i+1})^k,\;&x\in\intcc{x_{n,i}+\frac{3\Delta x}{4},x_{n,i+1}}\\
0\quad&\text{otherwise}
\end{cases}\text{ if $k$ is even.}
\end{align*}
This is the same construction as in Theorem 1, except that only one of the basic units is used.

The discussion about the continuity classification of $f_n$ holds, with the ``bad'' set
\[S_n=\left\{x_{n,i}\,,\,x_{n,i}+\frac{\Delta x}{4}\,,\,x_{n,i}+\frac{\Delta x}{2}\,,\,x_{n,i}+\frac{3\Delta x}{4}\,,\,x_{n,i+1}\right\}\]
of size 5; if $k$ is even, the midpoint is analytic as well. The comments on the bounds on derivatives hold.

$I_n(f_n)$ remains 0 as all evaluation points are 0, and the exact same calculation as before shows that
\[I(f_n)=\int_{a}^{b}f_n(x)\,\mathrm{d}x=\int_{x_{n,i}}^{x_{n,i+1}}f_n(x)\,\mathrm{d}x=\frac{(\Delta x)^{k+1}}{4^k}\text{.}\]
With $n$ points to be chosen in $[a,b]$ it is clear that $\Delta x$ is bounded below by $(b-a)/(n+1)$; an attempt to pick the midpoint $n$ times runs afoul of the fact that the error for the endpoints is even worse. Hence we have that
\[\abs{I_n(f_n)-I(f_n)}\geq\frac{(b-a)^{k+1}}{4^k}(n+1)^{-(k+1)}\in\Theta(n^{-(k+1)})\]
with order at most $k+1$, and the theorem is proved.
\end{proof}

We can tighten this result to requiring 3 (or 2, for even $k$) bad points by noticing that the endpoints $x_{n_i},x_{n,i+1}$ do not need to be bad. Instead we can use the analytic continuation of $f_n$ beyond these points (the monomial $(x-x_{i})^{k+1}$), and assume that $I_n$ can exactly integrate the rest of the function; the one piece between those two points alone makes the lower bound.

We may also keep the quadrature result 0 by making the $k+1$th derivative go negative quickly and smoothly connect to the zero function; this however necessitates breaking analyticity.

\section{Discussion and Conclusions.}
It may be asked next, can there be a finite, fixed set of bad points in the function, and how much improvement does that generate? We consider this moot, as when the bad points are fixed, the usual methods may be employed to simply \emph{find} them all, so that we may segment the region of integration at these points. Thus, all problems relating to having less than $N$ derivatives completely vanish.

The main implication of these theorems is that it is asymptotically fruitless to use any method that has 2 degrees or more beyond the differentiability class of the integrand: Theorem \ref{t2} proves the extra degrees have no effect.

An interesting observation is of the non-local effect of a single point of non-niceness destroying convergence everywhere; this has been noted in interpolation theory, under the name of the \emph{principle of contamination} \cite{saff}. Given the strong connection between interpolation polynomials and numerical integration, seeing the principle apply here is hardly a surprise.

Another insight that drops out of the details is a heuristic to minimize the impact of this lower bound: minimize the maximal distance between two evaluation points. That equally spaced points are optimal against this adversary speaks to the difference between optimizing for degrees of exactness and optimizing for absolute error under only $N$ derivatives.

There seems to be a penalty in the error bound, for example found in \cite{cruz}, for achieving a higher degree of exactness, if the extra derivatives for the higher order bound do not exist.

Thus of course the high degree Newton-Cotes rules retain their severe disadvantages, and we still do not recommend them; rather that a highly composite trapezoid or Simpson's rule be considered instead of the theoretically degree-superior Gauss-Legendre or Curtis-Clenshaw families, even at equal degree.

For future work we would like to know if the order $N+1$ in Theorem \ref{t2} is sharp, or if we can find a way to force order $N$ convergence with a fixed finite number of bad points.
\pagebreak
\paragraph{Acknowledgments.}
The author would like to thank Rajesh Pereira, University of Guelph, for help on formatting and literature search. This work was supported by a Natural Sciences and Engineering Research Council of Canada Postgraduate Scholarship.


\begin{thebibliography}{9}
\bibitem{stroud}
Arthur~H.~Stroud. ``Estimating quadrature errors for functions with low continuity''. \emph{SIAM Journal on Numerical Analysis} 3(3), 420--424, 1966. \doi{10.1137/0703036}

\bibitem{engel}
Johann Engelbrecht, Igor Fedotov, Tanya Fedotova, and Ansie Harding. ``Error bounds for quadrature methods involving lower order derivatives''. \emph{International Journal of Mathematical Education in Science and Technology} 34(6), 831--846, 2003. \doi{10.1080/00207390310001595429}

\bibitem{cruz}
David Cruz-Uribe and C. J. Neugebauer. ``Sharp error bounds for the trapezoidal rule and Simpson's rule''. \emph{Journal of Inequalities in Pure and Applied Mathematics} 3(4), art. 49, 2002. \xurl{http://www.emis.de/journals/JIPAM/article201.html}

\bibitem{ujevic}
Nenad Ujevi\'{c}. ``Error inequalities for a quadrature formula and applications''. \emph{Computers \& Mathematics with Applications} 48(10--11), 1531--1540, 2004. \doi{10.1016/j.camwa.2004.05.007}

\bibitem{huy}
Vu Nhat Huy and Qu\^{o}'c-Anh Ng\^{o}. ``New inequalities of Simpson-like type involving $n$ knots and the $m$th derivative''. \emph{Mathematical and Computer Modelling} 52(3--4), 522--528, 2010. \doi{10.1016/j.mcm.2010.03.049}

\bibitem{saff}
Edward B. Saff. ``A principle of contamination in best polynomial approximation''. In G\'{o}mez-Fernandez,~J.A., Guerra-V\'{a}zquez,~F., L\'{o}pez-Lagomasino,~G., Jim\'{e}nez-Pozo,~M.A. (eds.) \emph{Approximation and Optimization}, pp.79--97. Springer, Heidelberg, 1987. \doi{10.1007/BFb0089584}
\end{thebibliography}
\end{document}